\documentclass[10pt,leqno]{article}

\usepackage{amsmath,amssymb,amsthm,mathrsfs,dsfont}

\usepackage[margin=3cm]{geometry} %宽版时用，配合双倍行距

\usepackage{titlesec,hyperref}

\usepackage{graphicx}

\usepackage{color}
\usepackage{fancyhdr}
\titleformat{\subsection}{\it}{\thesubsection.\enspace}{1pt}{}

\newtheorem{theo}{Theorem}[section]
\newtheorem{lemm}[theo]{Lemma}
\newtheorem{defi}[theo]{Definition}

\newtheorem{prop}[theo]{Proposition}
\newtheorem{rema}[theo]{Remark}
\newtheorem{exam}{Example}
\numberwithin{equation}{section}

\allowdisplaybreaks %允许公式分页显示

\newcommand\ep{{\varepsilon}} %可定义一个简单符号来代替很长而常用的命令。

\begin{document}
\title{The well-posedness and blow up phenomenon for a Tsunamis model with time-fractional derivative
\hspace{-4mm}
}

\author{Bingbing $\mbox{Dai}^1$ \footnote{Email: daibb3@mail2.sysu.edu.cn},\quad
Wei $\mbox{Luo}^1$ \footnote{E-mail:  luowei23@mail2.sysu.edu.cn}, \quad
	Zhaoyang $\mbox{Yin}^{1,2}$\footnote{E-mail: mcsyzy@mail.sysu.edu.cn} \quad and
	Pei $\mbox{Zheng}^{1}$\footnote{E-mail: zhengp25@mail2.sysu.edu.cn}\\
	$^1\mbox{Department}$ of Mathematics,
	Sun Yat-sen University, Guangzhou 510275, China\\
	$^2\mbox{School}$ of Science,\\ Shenzhen Campus of Sun Yat-sen University, Shenzhen 518107, China}

\date{}
\maketitle
\hrule

\begin{abstract}
 This paper is concerned with the well-posedness of a time-fractional shallow-water equations, which has received little attention. In the realm of fractional calculus, numerous types of fractional derivatives have been explored in the literature.  Among these, one of the most notable and well-structured ones is the conformable fractional derivative. In this paper, we delve into the local well-posedness of the fractional tsunami shallow-water mathematical model in the critical Besov space $B^{\frac{3}{2}}_{2,1}$.  Under some symmetric and sign conditions, we show that the strong solution will blow up in finite time. \\
\vspace*{5pt}

\vspace*{5pt}
\end{abstract}

\vspace*{10pt}

%\phantomsection
%\addcontentsline{toc}{section}{\contentsname}
%添加目录到书签
\tableofcontents

\section{Introduction}

In this paper, we study the Cauchy problem of the time-fractional tsunami couple system(TFTCS)\cite{K-2024}:
     \begin{equation}\label{1.1}
     	\left\{
     	\begin{aligned}
     		&D^{\beta}_{t}u+u\partial_{x}u+g\partial_{x}\psi=0,\\
     		&D^{\beta}_{t}\psi+\partial_{x}[(\theta+\psi)u]=0, \\
     &u(x,0)=u_0, ~~~\psi(x,0)=\psi_0,
     	\end{aligned}
     	\right.
     \end{equation}
where $u(x,t)$ denotes the velocity of the water wave, $\psi(x,t)$ is referred to as the wave magnification factor, $\theta(x)$ is the ocean depth which is considered as a variable, $x$ refer to the distance, $t$ is the time variable ($t>0$), $g$ is the acceleration of gravity, and $\beta$ is the order of the fractional derivative ($0<\beta\le1$).

If $g=0$ and $\beta=1$, the first equation of \eqref{1.1} is the famous Burgers equation. The Burgers equation is the basic example of a nonlinear evolution equation leading to
shocks\cite{R-1970}.  A. Kiselev, F. Nazarov and  R. Shterenberg\cite{K-2008} studied about the regularity for fractal Burgers equation. C. Miao and G. Wu\cite{M-2009} established the well-posedness of the Burgers equation in critical Besov spaces $B^{1+\frac{1}{p}}_{p,1}$ with $p\in[1,+\infty)$.  L. Molinet, D. Pilodb and S. Vento\cite{M-2018} considered the well-posedness of some dispersive perturbations of the Burgers equation in $H^s$ with $s>\frac{3}{2}$. Recently, J. Li, Y. Yu and  W. Zhu\cite{L-2024} showed that the Burgers equation is ill-posedness in Sobolev spaces $H^s$ with $1\leq s<\frac{3}{2}$.

In the case $\beta=1$, the system \eqref{1.1} is a standard form of the hyperbolic conservation laws which has been studied widely. The classical result was obtained by P. Lax\cite{L-1957}. P. Lax considered the weak solution of a hyperbolic system and generalized the entropy condition and developed a theory of shocks. The stability of shock waves was proved by T.-P. Liu and Z. Xin\cite{LX-1992}. For the general well-posedness theory, we refer to \cite{K-1975,LY-1999}.  J. Xu and S. Kawashima\cite{X-2014} studied about the local well-posedness for the hyperbolic balance laws in the critical Besov space $B^{\frac{d}{2}+1}_{2,1}$.

Tsunamis are a natural phenomenon, characterized by their long wavelengths and immense destructive power. Researchers have indeed delved into various mathematical models to better understand and predict the behavior of tsunami waves. While much of the focus has traditionally been on integer-order models, there's been a growing interest in exploring fractional-order models as well. Integer-order models typically describe tsunami propagation using systems of partial differential equations that represent the dynamics of water motion and the interaction between the waves and the ocean floor. These models often incorporate principles of fluid dynamics, such as the shallow water equations, to simulate the evolution of tsunami waves over space and time\cite{KC-2017,MM-2018,MF-2022}. However, fractional-order models offer a more nuanced approach by considering derivatives of non-integer order, which can capture more accurately the complex behavior of tsunami waves, particularly in regions with variable bathymetry or irregular coastlines. These models have the potential to provide deeper insights into phenomena like wave dispersion, dissipation, and nonlinear interactions, which are crucial for improving tsunami forecasting and risk assessment.

\subsection{Fractional calculus}

Fractional calculus is a fascinating field within mathematical analysis that delves into the realm of noninteger order powers of differentiation and integration operators. Compared to integer-order differential equations, fractional differential equations offer significant opportunities in recent times due to the remarkable and insightful outcomes achieved by employing fractional operators to model real-world phenomena across a multitude of disciplines such as physics, chemistry, biology, engineering, and beyond\cite{A-2013,B-2016,C-1997,M-1996}. This growing importance and popularity stem from its ability to provide nuanced insights and solutions to complex problems\cite{P-1999,S-1993}, making it an indispensable tool for researchers and practitioners alike. There are various definitions of fractional derivatives. Two of which are the most popular ones. \\
(i)  Riemann-Liouville derivatives\cite{M-1993, S-1993}. For $\beta\in(0,1]$, the $\beta$ derivative of $f$ is
$$ D^\beta_af(t):=\frac{1}{\Gamma(\beta)}\frac{d}{dt}\int^t_a\frac{f(x)}{(t-x)^{\beta}}dx .$$
(ii) Caputo derivatives\cite{K-2006, P-1999}. For $\beta\in(0,1]$, the $\beta$ derivative of $f$ is
$$ D^\beta_af(t):=\frac{1}{\Gamma(\beta)}\int^t_a\frac{f'(x)}{(t-x)^{\beta}}dx .$$
Recently, there are a lot of papers studied about the time-fractional parabolic partial differential equations\cite{A-2016,A-2019, M-2024,W-2024,G-2024,L-2022,H-2022}. The key idea of them was to use the fixed point theory to establish the existence and uniqueness in suitable Banach spaces. However, this method was invalid for the hyperbolic type system. The fractional derivatives defined by the above way do not obey Leibniz's law and the chain rule. It is hard to establish the well-posedness for a hyperbolic system with Riemann-Liouville derivatives or Caputo derivatives.

In order to study \eqref{1.1}, we consider a well-behaved simple fractional derivative depending only on the basic limit definition of the derivative\cite{K-2014}.
\begin{defi}[The conformable fractional derivative\cite{K-2014}]\label{Def1}
For a function $f:[0,\infty)\to \mathbb{R}$ the $\beta$ derivative of $f$ at $t>0$ was defined by
\begin{equation}
     	D^{\beta}_{t}f(t)=\lim\limits_{\varepsilon\to0}\frac{f(t+\varepsilon t^{1-\beta})-f(t)}{\varepsilon},
\end{equation}
where $0<\beta\le1$. If $f$ is $\beta-$differentialbe in $(0, a), a>a$ and  $\lim\limits_{t\to0}(D^{\beta}_{t}f)(t)$ exists, then the fractional derivative at $0$ is defined as $D^{\beta}_{t}f(0)=\lim\limits_{t\to0}(D^{\beta}_{t}f)(t)$.
\end{defi}
It is easy to check that the conformable fractional derivative satisfy the known formula of the derivative of the product of two functions\cite{A-2015}:
    \begin{equation}
    	D^{\beta}_{t}(fg)=fD^{\beta}_{t}g+gD^{\beta}_{t}f.
    \end{equation}
In the following sequel, the notation $D^{\beta}_{t}$ was defined by the conformable fractional derivative.
\subsection{Main results}
Without loss of generality, we can take $g=1$. To the best of our knowledge, the mathematical analysis result for Eqs.\eqref{1.1} has not yet been studied. Firstly we need to transfer Eqs.\eqref{1.1} to a quasilinear symmetric system. For this purpose, we introduce the new un-know function $v=2\sqrt{\psi+\theta}$, then it can be seen that $(u,v)$ satisfies
     \begin{equation}\label{1.3}
     	\left\{
     	\begin{aligned}
     		&D^{\beta}_{t}u+u\partial_{x}u+\frac{1}{2}v\partial_{x}v=\partial_{x}\theta,\\
     		&D^{\beta}_{t}v+\frac{1}{2}v\partial_{x}u+u\partial_{x}v=0,\\
     		&u(x,0)=u_0(x), ~~~v(x,0)=2\sqrt{\psi_0(x)+\theta(x)}.
     	\end{aligned}
     	\right.
     \end{equation}
     Moreover, we write Eqs.\eqref{1.3} as follows:
     \begin{equation}\label{1.4}
     	\left\{
     	\begin{aligned}
     		&D^{\beta}_{t}U+A(U)\partial_{x}U=M(x),\\
     		&U(x,0)=U_0,
     	\end{aligned}
     	\right.
     \end{equation}
     where $U=\left(
     \begin{matrix}u\\
     	v\end{matrix}
     \right)$, $A(U)=A(u,v)=\left(
     \begin{matrix}u &\frac{1}{2}v\\
     	\frac{1}{2}v &u\end{matrix}
     \right)$, $M(x)=\left(
     \begin{matrix}\partial_x\theta\\
     	0\end{matrix}
     \right)$ and $U_0=\left(
     \begin{matrix}u_0(x)\\
     	2\sqrt{\psi_0(x)+\theta(x)}\end{matrix}
     \right)$.

   It is worth mentioning that the novelty of this study is to solve the system with time-fractional derivative and obtain the existence and uniqueness of the solution in critical Besov space.

   Our main theorems can be stated as follows:
   \begin{theo}[Local well-posedness]\label{them1}
   Let $U_0$ be in $(B^{\frac{3}{2}}_{2,1})^2$, $\theta\in B^{\frac{3}{2}}_{2,1}$ and $\beta=1$. Then there exists a time $T=T(U_0,\theta)>0$ such that \eqref{1.4} has a unique solution $U$ in $\bigg(\mathcal{C}([0,T],B^{\frac{3}{2}}_{2,1})\cap\mathcal{C}^1([0,T],B^{\frac{1}{2}}_{2,1})\bigg)^2$. Moreover, the solution depends continuously on the initial data.
   \end{theo}

\begin{theo}[Local well-posedness for $0<\beta<1$]\label{them2}
	Let $U_0$ be in $(B^{\frac{3}{2}}_{2,1})^2$, $\theta\in B^{\frac{3}{2}}_{2,1}$ and $0<\beta<1$. Then there exists a time $T=T(\beta, U_0, \theta)>0$ such that \eqref{1.4} has a unique solution $U$ in $\bigg(\mathcal{C}([0,T],B^{\frac{3}{2}}_{2,1})\cap\mathcal{C}^\beta([0,T],B^{\frac{1}{2}}_{2,1})\bigg)^2$.
\end{theo}

\begin{rema}
If the initial data and the function $\theta$ belong to $H^s$ with $s>\frac{3}{2}$, we can also obtain that the solution belongs to $\bigg(\mathcal{C}([0,T],H^s)\cap\mathcal{C}^\beta([0,T],H^{s-1})\bigg)^2$. The proof is similar to theorem \ref{them2}, we omit the details here.
\end{rema}

\begin{theo}[Blow up phenomenon]\label{them3}
Let $(u_0,\psi_0)\in (H^3)^2$ and $\theta\in H^3$. Assume that $u_0$ is odd, $\psi_0, \theta$ are even and $$\psi_0(0)=\theta(0)=0,\psi''_0(0)\leq 0, \theta''(0)\geq 0, u'_0(0)<0.$$
Then the corresponding solution of \eqref{1.1} will blow up in finite time.
\end{theo}

 This paper is aimed to establish the locally well-posedness in the critical Besov space $B^{\frac{3}{2}}_{2,1}$. The content of this paper is the following. In Section 2, we recall some basic definitions and the related results about Besov spaces. In Section 3, we prove the local well-posedness of \eqref{1.4} in critical Besov space $B^{\frac{3}{2}}_{2,1}$ with $\beta=1$ and present the blow-up criterion. In Section 4, we consider the case $0<\beta<1$. Section 5 devote to investigating the blow up phenomenon.
\section{Preliminaries}
  ~~~~In this section, we will recall the theory about Besov spaces and some properties about fractional calculus, which will be used in the sequel. Firstly, we state Bernstein's inequalities.
  \begin{prop}\cite{B-2011}\label{Bernstein}
  	Let $\mathscr{C}$ be an annulus and $\mathscr{B}$ be a ball. A constant $C\textgreater0$ exists such that for any $k\in\mathbb{N}$, $1\leq p\leq q\leq \infty$, and any function $f\in L^p(\mathbb{R})$, we have \\
  	$$Supp(\mathscr {F}f) \subset \lambda\mathscr{B} \Longrightarrow \|D^k f\|_{L^p}=\sup\limits_{|\alpha|=k}\|\partial^{\alpha}f\|_{L^q} \leq C^{k+1}{\lambda}^{k+d(\frac{1}{p}-\frac{1}{q})}\|f\|_{L^p}$$
  		$$Supp(\mathscr {F}f) \subset \lambda\mathscr{C} \Longrightarrow  C^{-k-1}{\lambda}^{k}\|f\|_{L^p} \leq \|D^k f\|_{L^p} \leq C^{k+1}{\lambda}^{k}\|f\|_{L^p}$$
  \end{prop}

  	Then we recall some basic properties on the Littlewood-Paley theory.
  \begin{prop}\label{prop}\cite{B-2011}
  		Let $\mathcal{C}$ be the annulus $\{\xi\in\mathbb{R}^d:\frac 3 4\leq|\xi|\leq\frac 8 3\}$ and $\mathcal{B}$ be the ball $\{\xi\in\mathbb{R}^d:|\xi|\leq\frac 4 3\}$. There exists a couple of functions $(\chi, \varphi)$, valued in the interval $[0,1]$, belonging respectively to $\mathcal{D}(\mathcal{B})$ and $\mathcal{D}(\mathcal{C})$, and such that
  	$$ \forall\xi\in\mathbb{R}^d,\ \chi(\xi)+\sum_{j\ge0}\varphi(2^{-j}\xi)=1, $$
  	$$ |j-j'|\geq 2\Rightarrow\mathrm{Supp}\ \varphi(2^{-j}\cdot)\cap \mathrm{Supp}\ \varphi(2^{-j'}\cdot)=\emptyset. $$
  		$$ j\geq 1\Rightarrow\mathrm{Supp}\ \chi(\cdot)\cap \mathrm{Supp}\ \varphi(2^{-j}\cdot)=\emptyset. $$
  	Then for all $u\in \mathcal{S}'(\mathbb{R}^d)$, define $\Delta_ju=0$ for $j\le-2$; $\Delta_{-1}u\triangleq\chi(D)u=\mathcal{F}^{-1}(\chi\mathcal{F} u)$;  $\Delta_{j}u\triangleq\varphi(2^{-j}D)u=\mathcal{F}^{-1}(\varphi(2^{-j}\xi)\mathcal{F} u)$ for $j\le-2$;
  	Hence $u=\sum_{j\in\mathbb{Z}}\Delta_{j}u$, in $\mathcal{S}'(\mathbb{R}^d)$.
  	  \end{prop}

  	The nonhomogeneous Bony's decomposition is defined by $uv=T_uv+T_vu+R(u,v)$ with
  	$$
  T_uv=\sum_{j}S_{j-1}u\Delta_{j}u,~~~~~R(u,v)=\sum_{j}\sum_{|j'-j|\le1}\Delta_{j}u\Delta_{j'}u.
  	$$
  	
  	Next we introduce some properties about Besov spaces.

  \begin{prop}\label{embedding}\cite{B-2011}
  	Let $1\le p_1\le p_2\le \infty$ and $1\le r_1\le r_2\le \infty$, and let $s$ be a real number. Then we have
  	$$
  	B^s_{p_1,r_1}\hookrightarrow B^{s-d(1/p_1-1/p_2)}_{p_2,r_2}.
  	$$
  	If $s>\frac{d}{p}$ or $s=\frac{d}{p}, r=1$, we then have
  	$$
  	B^s_{p,r}\hookrightarrow L{^{\infty}}.
  	$$
  \end{prop}
  \begin{prop}\cite{B-2011}
  	The set $B^s_{p,r}$ is a Banach space and satisfies the Fatou property, namely, if $(u_n)_{n\in \mathbb{N}}$
  	is a bounded sequence in $B^s_{p,r}$, then an element $u\in B^s_{p,r}$ and a subsequence $(u_{n_k})_{k\in\mathbb{N}}$ exist such that
  	\begin{equation*}
  		\lim\limits_{k\to\infty}u_{n_k}=u~~~in ~\mathcal{S}'~~~~~and~~~~\|u\|_{B^s_{p,r}}\le C\liminf\limits_{k\to\infty}\|u_{n_k}\|_{B^s_{p,r}}.
  	\end{equation*}
  \end{prop}
  	%\begin{prop}\cite{B-2011}
  %	Let $s\in\mathbb{R},\ 1\leq p,r\leq\infty.$
  %	\begin{equation*}\left\{
  	%	\begin{array}{l}
  	%		B^s_{p,r}\times B^{-s}_{p',r'}\longrightarrow\mathbb{R},  \\
  	%		(u,\phi)\longmapsto \sum\limits_{|j-j'|\leq 1}\langle \Delta_j u,\Delta_{j'}\phi\rangle,
  	%	\end{array}\right.
  	%\end{equation*}
  	%defines a continuous bilinear functional on $B^s_{p,r}\times B^{-s}_{p',r'}$. Denoted by $Q^{-s}_{p',r'}$ the set of functions $\phi$ in $\mathcal{S}'$ such that
  	%$\|\phi\|_{B^{-s}_{p',r'}}\leq 1$. If $u$ is in $\mathcal{S}'$, then we have
  	%$$\|u\|_{B^s_{p,r}}\leq C\sup_{\phi\in Q^{-s}_{p',r'}}\langle u,\phi\rangle.$$
  %\end{prop}
  \begin{lemm}\label{product}\cite{B-2011,L-2016}
  	For $s=\frac{1}{2}$, $p=2$ and $r=1$, the space $L^{\infty} \cap B^\frac{1}{2}_{2,1}$ is an algebra, and a constant $C=C(s,d)$ exists such that
  	$$ \|uv\|_{B^{\frac{1}{2}}_{2,1}}\leq C(\|u\|_{L^{\infty}}\|v\|_{B^{\frac{1}{2}}_{2,1}}+\|u\|_{B^{\frac{1}{2}}_{2,1}}\|v\|_{L^{\infty}})\le C\|u\|_{B^{\frac{1}{2}}_{2,1}}\|v\|_{B^{\frac{1}{2}}_{2,1}}. $$
  \end{lemm}
  Let's review the Gronwall lemma as follows.
  \begin{lemm}\label{osgood}\cite{B-2011}
  	Let $f(t),~ g(t)\in C^{1}([0,T]), f(t), g(t)\geq 0.$ Let $h(t)$ be a continuous function on $[0,T].$ Assume that, for any $t\in [0,T]$ such that
  	$$\frac 1 2 \frac{d}{dt}f^{2}(t)\leq h(t)f^{2}(t)+g(t)f(t).$$
  	Then for any time $t\in [0,T],$ we have
  	$$f(t)\leq f(0)\exp\big(\int_0^th(\tau)d\tau\big)+\int_0^t g(\tau)\exp\big(\int_\tau ^t h(\tau)dt'\big)d\tau.$$
  \end{lemm}
  Now we state some useful results in the linear symmetric system, which are important to the proofs of our main theorem later.
  \begin{equation}\label{transport}
  	\left\{\begin{array}{l}
  		\partial_tU+\mathcal{A}\partial_xU=F, \\
  		U|_{t=t_0}=U_0,
  	\end{array}\right.
  \end{equation}
  where $I$ is an interval of $\mathbb{R}$, $\mathcal{A}$ is a smooth bounded function from $I\times\mathbb{R}$ into the space of $N\times N$ symmetric matrices with real coefficients, $U_0:\mathbb{R}^1\to\mathbb{R}^N$, $F:I\times\mathbb{R}^1\to\mathbb{R}^N$ and $t_0\in I$.

  \begin{lemm}\label{existence}\cite{B-2011}
  	Let $s>0$,$r\in[1,\infty]$, and $V$ satisfy
  	\begin{equation*}
  		\partial_tV+\sum_{k=1}^{d}\mathcal{A}_k\partial_kV=F.
  	\end{equation*}
  	 Let $V_j\triangleq\Delta_jV$, $\overline{S}_j\triangleq S_j$ if $j\ge0$, and $\overline{S}_j\triangleq\Delta_{-1}$ if $j\in\{-2,-1\}$. We have
  	 	\begin{equation*}
  	 	\partial_tV_j+\sum_{k=1}^{d}(\overline{S}_{j-1}\mathcal{A}_k)\partial_kV_j=\Delta_jF+R_j, ~~for~all~j\ge-1,
  	 \end{equation*}
  	 where $R_j$ satisfies, for $\forall~t\in I$,
  	 \begin{equation}
  	 	2^{js}||R_j^n||_{L^2}\le Cc_j(t)\sum_{k=1}^{d}\bigg(||\nabla \mathcal{A}_k(t)||_{L^{\infty}}||\nabla V(t)||_{B^{s-1}_{2,r}}+||\nabla V(t)||_{L^{\infty}}||\nabla \mathcal{A}_k(t)||_{B^{s-1}_{2,r}}\bigg).
  	 \end{equation}
  	%If $0<s<d/2+1$, then we also have
  	%\begin{equation}
  	%2^{js}||R_j^n||_{L^2}\le Cc_j(t)||\nabla V(t)||_{B^{s-1}_{2,r}}\sum_{k=1}^{d}||\nabla \mathcal{A}_k(t)||_{L^{\infty}\cap B^{d/2}_{2,\infty}},
  	%\end{equation}
  	If $s=d/2+1$, then we have, for all $\varepsilon>0$,
  	\begin{equation}
  		2^{j(d/2+1)}||R_j^n||_{L^2}\le Cc_j(t)||\nabla V(t)||_{B^{d/2}_{2,r}}\sum_{k=1}^{d}||\nabla \mathcal{A}_k(t)||_{L^{\infty}\cap B^{d/2+\varepsilon}_{2,\infty}}.
  	\end{equation}
  	
  \end{lemm}
  \begin{lemm}\label{priori estimate}\cite{B-2011}
  	Let $s=\frac{3}{2},\ r=1,\ U_0\in B^{\frac{3}{2}}_{2,1}$, and $F\in \mathcal{C}(I;B^{\frac{3}{2}}_{2,1})$.
  	Suppose that the matrices $\mathcal{A}$ are symmetric and continuous with respect to $(t,x)$, and that $\partial_x \mathcal{A}\in\mathcal{C}(I;B^{\frac{1}{2}}_{2,1})$.

 %--$\nabla \mathcal{A}_k\in\mathcal{C}(I;B^{d/2+\varepsilon}_{2,\infty})$ for some  $\varepsilon>0$ if $s=d/2+1$ and $r>1$,

% --$\nabla \mathcal{A}_k\in\mathcal{C}(I;B^{d/2}_{2,\infty}\cap L^{\infty})$ if $0<s<d/2+1$.

  The system
  \begin{equation}\label{transport}
  	\left\{\begin{array}{l}
  		\partial_tU+\mathcal{A}\partial_xU=F, \\
  		U|_{t=t_0}=U_0,
  	\end{array}\right.
  \end{equation}
  then has a unique solution $U\in\mathcal{C}(I;B^{\frac{3}{2}}_{2,1})\cap \mathcal{C}^1(I;B^{\frac{1}{2}}_{2,1})$. Moreover, for all $t\in I$ and $C=C(d,s)$, we have
  $$
  |U(t)|_{B^{\frac{3}{2}}_{2,1}}\le |U_0|_{B^{\frac{3}{2}}_{2,1}}\exp\bigg(\int_{0}^{t}C\|\partial_x\mathcal{A}(t')\|_{B^{\frac{1}{2}}_{2,1}}dt'\bigg)+\int_{0}^{t}|F'(t')|_{B^{s}_{2,r}}\exp\bigg(\int_{t'}^{t}C\|\partial_x  \mathcal{A}(t'')\|_{B^{\frac{1}{2}}_{2,1}}dt''\bigg)dt'
  $$
  with $|U(t)|_{B^{\frac{3}{2}}_{2,1}}\triangleq||2^{\frac{3}{2}s}(\|\Delta_ju\|_{L^2}+\|\Delta_jv\|_{L^2})||_{l^1}$.
  %and
%  	\begin{equation*}
 % 		a_s(t)\triangleq
  %		\left\{\begin{array}{l}
  	%	\sum_{k}\|\nabla \mathcal{A}_k(t)\|_{B^{s-1}_{2,r}},\ \text{if}~  s>d/2+1,~ \text{or}~s=d/2+1 \text{and}~r=1, \\
  	%	\sum_k\|\nabla \mathcal{A}_k(t)\|_{B^{d/2+\varepsilon}_{2,\infty}},\ \text{if}\ s=d/2+1 \text{and}\ r>1, \\
  	%	\sum_k\|\nabla \mathcal{A}_k(t)\|_{B^{d/2}_{2,\infty}\cap L^{\infty}},\ \text{if}\ 0<s<d/2+1.
  	%	\end{array}\right.
  	%\end{equation*}
  \end{lemm}
   \begin{lemm}\label{2continuity}
  	 Let ~$\bar{\mathbb{N}}=\mathbb{N}\cup\{\infty\}$. For $d=1$, consider a sequence $(\mathcal{A}^n)_{n\in\bar{\mathbb N}}$ of continuous bounded functions on $I\times\mathbb{R}$ with values in the set of symmetric $N\times N$ matrices. Assuming that $\partial_x\mathcal{A}^n\in \mathcal{C}([0,T];B^{\frac{1}{2}}_{2,1})$ and there exists a nonnegative integral function $\alpha$ over $I$ such that
  \begin{equation}\label{eqq3}
  \|\partial_x\mathcal{A}^n\|_{B^{\frac{1}{2}}_{2,1}}\le\alpha(t)~~~~for~all~t\in I
  \end{equation}
   and
   \begin{equation}\label{eqq4}
  \mathcal{A}^n-\mathcal{A}^{\infty}\rightarrow 0 ~~as ~n\to\infty,~~~~in~L^1(I,B^{\frac{1}{2}}_{2,1}).
  \end{equation}
   Define by $V^n$ the solutions of
  	\begin{equation}\label{transport}
  	\left\{\begin{array}{l}
  		\partial_tV^n+\mathcal{A}^n\partial_xV^n=F, \\
  		V^n|_{t=t_0}=V_0,
  	\end{array}\right.
  \end{equation}
  	with $F\in\mathcal{C}(I,B^{\frac{1}{2}}_{2,1}),\ V_0\in B^{\frac{1}{2}}_{2,1}$, then $(V^n)_{n\in\bar{\mathbb N}}$ converges to $V^{\infty}$ in $\mathcal{C}([0,T];B^{\frac{1}{2}}_{2,1})$.
  \end{lemm}
  \begin{proof}\
\quad Let $W^n=V^n-V^{\infty}$. Then we have
$$
\partial_tW^n+\mathcal{A}^n\partial_xW^n=(\mathcal{A}^{\infty}-\mathcal{A}^n)\partial_xV^{\infty}.
$$
Let us make the additional assumption that $F\in\mathcal{C}(I,B^{\frac{3}{2}}_{2,1}),\ V_0\in B^{\frac{3}{2}}_{2,1}$. As $V^n(0)=V^{\infty}(0)$, applying Lemma \ref{priori estimate} and \eqref{eqq3}, we have
$$
  \|W^n(t)\|_{B^{\frac{1}{2}}_{2,1}}\le\int_{0}^{t}\|(\mathcal{A}^{\infty}-\mathcal{A}^n)\partial_xV^{\infty}\|_{B^{\frac{1}{2}}_{2,1}}e^{\int_{t'}^{t}C\alpha(t'')dt''}dt'.
$$
Then by lemma \ref{product}, we get
$$
  \|W^n(t)\|_{B^{\frac{1}{2}}_{2,1}}\le\int_{0}^{t}\|\mathcal{A}^{\infty}-\mathcal{A}^n\|_{B^{\frac{1}{2}}_{2,1}}\|\partial_xV^{\infty}\|_{B^{\frac{1}{2}}_{2,1}}e^{\int_{t'}^{t}C\alpha(t'')dt''}dt'.
$$
Taking advantage of \eqref{eqq4}, it is easy to conclude that $V^n$ tends to $V^{\infty}$ in $\mathcal{C}([0,T];B^{\frac{1}{2}}_{2,1})$.

To take the non-smooth case $F\in\mathcal{C}(I,B^{\frac{1}{2}}_{2,1})$, and $V_0\in B^{\frac{1}{2}}_{2,1}$. We can write that

\begin{equation}\label{eqqqq5}
\|W^n(t)\|_{B^{\frac{1}{2}}_{2,1}}\le\|V^n(t)-V_j^n(t)\|_{B^{\frac{1}{2}}_{2,1}}+\|V_j^n(t)-V_j^{\infty}(t)\|_{B^{\frac{1}{2}}_{2,1}}+\|V_j^{\infty}(t)-V^{\infty}(t)\|_{B^{\frac{1}{2}}_{2,1}},
\end{equation}
where $V_j^n$ satisfy the following equation
\begin{equation*}
  	\left\{\begin{array}{l}
  		\partial_tV_j^n+\mathcal{A}^n\partial_xV_j^n=\mathbb{E}_jF, \\
  		V_j^n|_{t=t_0}=\mathbb{E}_jV_0.
  	\end{array}\right.
  \end{equation*}
Since $V^n-V_j^n$ solves
\begin{equation*}
  	\left\{\begin{array}{l}
  		\partial_tV+\mathcal{A}^n\partial_xV=F-\mathbb{E}_jF, \\
  		V|_{t=t_0}=V_0-\mathbb{E}_jV_0,
  	\end{array}\right.
  \end{equation*}
then by Lemma \ref{priori estimate} and \eqref{eqq3}, we have
$$
  \|V^n(t)-V_j^n(t)\|_{B^{\frac{1}{2}}_{2,1}}\le e^{\int_{t'}^{t}C\alpha(t'')dt''}\bigg(\|V_0-\mathbb{E}_jV_0\|_{B^{\frac{1}{2}}_{2,1}}+\int_{0}^{t}\|F-\mathbb{E}_jF\|_{B^{\frac{1}{2}}_{2,1}}dt'\bigg).
$$
By virtue of $\mathbb{E}_jV_0\to V_0$ and $\mathbb{E}_jF\to F$ in $B^{\frac{1}{2}}_{2,1}$, we can find some $j\in\bar{\mathbb N}$ such that $\|V^n(t)-V_j^n(t)\|_{B^{\frac{1}{2}}_{2,1}}\le \epsilon/3$. For fixed $j$, let $n$ tends to infinity so that the second term in the right-hand side of \eqref{eqqqq5} is less than $\epsilon/3$. Thus, we conclude that $(V^n)_{n\in\bar{\mathbb N}}$ converges to $V^{\infty}$ in $\mathcal{C}([0,T];B^{\frac{1}{2}}_{2,1})$.
\end{proof}
 Moreover, we state some properties of conformable fractional derivatives.
 \begin{lemm}\cite{K-2014}
 	Let $\beta\in]0,1]$ and $a,b\in\mathbb{R}$, then
 	$$
 	D_t^{\beta}(af+bg)=aD_t^{\beta}f+bD_t^{\beta}g,
 	$$
 	$$
 	D_t^{\beta}(fg)=fD_t^{\beta}g+gD_t^{\beta}f,
 	$$
 	$$
 	D_t^{\beta}(\frac{f}{g})=\frac{fD_t^{\beta}g+gD_t^{\beta}f}{g^2}.
 	$$

 \end{lemm}
 Let $I^0_{\beta}$ be the fractional integral given by
 $$I^0_{\beta}f(t)=\int_{0}^{t}x^{\beta-1}f(x)dx\quad \text{for} \quad \beta\in]0,1].$$
 The following lemma gives the relationship between $D_t^{\beta}$ and $I^0_{\beta}$.
 \begin{lemm}\label{fractional}\cite{K-2014}
 	Assume that $f:[a,\infty]\to\mathbb{R}$ be differentiable and $\beta\in]0,1]$. Then, for all $t>0$, we have
 	$$
 	I^0_{\beta}D_t^{\beta}(f)(t)=f(t)-f(0).
 	$$
 \end{lemm}

 {\bf Notation}. Since all function spaces in the following sections are over $\mathbb{R}$, for simplicity, we drop $\mathbb{R}$ in the notation of function spaces if there is no ambiguity.

\section{Local well-posedness for $\beta=1$}
~~~~In this section, we establish the local well-posedness of Eq.\eqref{1.4} in critical Besov spaces $B^{\frac{3}{2}}_{2,1}$ for $\beta=1$. We will prove the Theorem \ref{them1} in six steps.\\
$\bf{Step~1:~Approximate~Solutions}$

First, we construct approximate solutions which are smooth solutions of some quasilinear symmetric equations. Set $U^0=0$ and define a sequences $(U^n)_{n\in\mathbb N}$ of smooth functions by solving the following linear symmetric equations:
\begin{equation}\label{2.1}
	\left\{\begin{array}{l}
	\partial_{t}U^{n+1}+A(U^n)\partial_{x}U^{n+1}=M(x),\\
	U^{n+1}|_{t=0}=S_{n+1}U_0.
	\end{array}\right.
\end{equation}
Assuming that $U^n\in L^{\infty}([0,T],B^{\frac{3}{2}}_{2,1})$ for all $T>0$, we can check that $A(U^n)\in L^{\infty}([0,T],B^{\frac{3}{2}}_{2,1})$. Taking advantage of the theory of the linear symmetric system, we obtain that $U^{n+1}$ of \eqref{2.1} in $L^{\infty}([0,T],B^{\frac{3}{2}}_{2,1})$.\\
$\bf{Step~2:~Uniform~Bounds}$

Next, we are going to prove uniform estimates in $B^{\frac{3}{2}}_{2,1}$ for the approximate solution $U^n$. We claim that some constant $C_0$ can be found such that
\begin{equation}\label{2.2}
	 T<\min\bigg\{1,\frac{\rm{ln}2}{2C_0\big(||U_0||_{B^{\frac{3}{2}}_{2,1}}+C_0||\theta||_{B^{\frac{3}{2}}_{2,1}}\big)}\bigg\}\Longrightarrow||U^n||_{L^{\infty}([0,T];B^{\frac{3}{2}}_{2,1})}\le 2(||U_0||_{B^{\frac{3}{2}}_{2,1}}+C_0||\theta||_{B^{\frac{3}{2}}_{2,1}}).
\end{equation}
We will proceed by induction. As $U^0=0$, the inequality \eqref{2.2} holds true for $n=0$. Suppose that it holds for some $n$. We estimates $U^{n+1}$ by lemma \ref{existence}. Applying $\Delta_j (j\ge -1)$ to \eqref{2.1},  we get
\begin{equation*}
	\partial_t\Delta_j U^{n+1}+(\overline{S}_{j-1}A(U^n))\partial_x\Delta_j U^{n+1}=R_j^n+\Delta_j M,
\end{equation*}
where $R_j^n$ satisfies,
\begin{equation*}
	||R_j^n||_{L^2}\le Cc^n_j(t)\cdot 2^{-\frac{3}{2}j}\bigg(||\nabla U^n(t)||_{L^{\infty}}||\nabla U^{n+1}(t)||_{B^{\frac{1}{2}}_{2,1}}+||\nabla U^{n+1}(t)||_{L^{\infty}}||\nabla U^{n}(t)||_{B^{\frac{1}{2}}_{2,1}}\bigg),
\end{equation*}
with $||(c^n_j(t))||_{l^2}\le 1$. Combining the above inequality and the fact that
\begin{equation*}
	||\nabla\overline{S}_{j-1}A(U^n)||_{L^{\infty}}\le C||\nabla U^n||_{L^{\infty}},
\end{equation*}
we have
\begin{equation}\label{eq1}\begin{aligned}
	\frac{1}{2}\frac{d}{dt}||\Delta_j U^{n+1}||^2_{L^{2}}\le &C||\nabla U^n||_{L^{\infty}}||\Delta_j U^{n+1}||^2_{L^{2}}\\
&+C||R_j^n||_{L^2}||\Delta_j U^{n+1}||_{L^{2}}+||\partial_x\Delta_j\theta||_{L^{2}}||\Delta_j U^{n+1}||_{L^{2}}.
\end{aligned}\end{equation}
As $B^{\frac{1}{2}}_{2,1}\hookrightarrow L^{\infty}$, we get
\begin{equation*}\begin{aligned}
		\frac{1}{2}\frac{d}{dt}||\Delta_j U^{n+1}||^2_{L^{2}}\le &C||U^n||_{B^{\frac{3}{2}}_{2,1}}||\Delta_j U^{n+1}||^2_{L^{2}}\\
		&+\big(C c_j(t)2^{-\frac{3}{2}j}||U^n||_{B^{\frac{3}{2}}_{2,1}}||U^{n+1}||_{B^{\frac{3}{2}}_{2,1}}+Cc_j(t)2^{-\frac{3}{2}j}||\theta||_{B^{\frac{3}{2}}_{2,1}}\big)||\Delta_j U^{n+1}||_{L^{2}}.
\end{aligned}\end{equation*}
Let $\varepsilon>0$, then from the previous inequality, we infer that
\begin{equation*}\begin{aligned}
		\frac{d}{dt}\big(||\Delta_j U^{n+1}||^2_{L^{2}}+\varepsilon\big)^{\frac{1}{2}}\le &C||U^n||_{B^{\frac{3}{2}}_{2,1}}||\Delta_j U^{n+1}||_{L^{2}}\\
		&+C c_j(t)2^{-\frac{3}{2}j}(||U^n||_{B^{\frac{3}{2}}_{2,1}}||U^{n+1}||_{B^{\frac{3}{2}}_{2,1}}+||\theta||_{B^{\frac{3}{2}}_{2,1}}).
\end{aligned}\end{equation*}
Because $||\Delta_j U^{n+1}||_{L^{2}}\le c_j(t)2^{-\frac{3}{2}j}||U^{n+1}||_{B^{\frac{3}{2}}_{2,1}}$, we thus get
\begin{equation*}
	\frac{d}{dt}\big(||\Delta_j U^{n+1}||^2_{L^{2}}+\varepsilon\big)^{\frac{1}{2}}\le C c_j(t)2^{-\frac{3}{2}j}(||U^n||_{B^{\frac{3}{2}}_{2,1}}||U^{n+1}||_{B^{\frac{3}{2}}_{2,1}}+||\theta||_{B^{\frac{3}{2}}_{2,1}}),
\end{equation*}
which leads to
\begin{equation*}
	\big(||\Delta_j U^{n+1}||^2_{L^{2}}+\varepsilon\big)^{\frac{1}{2}}\le \big(||\Delta_j U_0||^2_{L^{2}}+\varepsilon\big)^{\frac{1}{2}}+C 2^{-\frac{3}{2}j}\int_0^t c_j(t')\big(||U^n||_{B^{\frac{3}{2}}_{2,1}}||U^{n+1}||_{B^{\frac{3}{2}}_{2,1}}+||\theta||_{B^{\frac{3}{2}}_{2,1}}\big)dt'.
\end{equation*}
Let $\varepsilon$ tends to 0, we end up with
\begin{equation*}
	||\Delta_j U^{n+1}||_{L^{2}}\le||\Delta_j U_0||_{L^{2}}+C 2^{-\frac{3}{2}j}\int_0^t c_j(t')\big(||U^n||_{B^{\frac{3}{2}}_{2,1}}||U^{n+1}||_{B^{\frac{3}{2}}_{2,1}}+||\theta||_{B^{\frac{3}{2}}_{2,1}}\big)dt'.
\end{equation*}
Next, we multiply both sides of the above inequality by $2^{\frac{3}{2}j}$ and then sum over $j$ to obtain
\begin{equation*}\begin{aligned}
		||U^{n+1}||_{B^{\frac{3}{2}}_{2,1}}\le ||U_0||_{B^{\frac{3}{2}}_{2,1}}+C\int_0^t\big(||U^n||_{B^{\frac{3}{2}}_{2,1}}||U^{n+1}||_{B^{\frac{3}{2}}_{2,1}}+||\theta||_{B^{\frac{3}{2}}_{2,1}}\big)dt'.
\end{aligned}\end{equation*}
Using the Gronwall lemma, we have that
\begin{equation*}\begin{aligned}
		||U^{n+1}||_{B^{\frac{3}{2}}_{2,1}}\le (||U_0||_{B^{\frac{3}{2}}_{2,1}}+Ct||\theta||_{B^{\frac{3}{2}}_{2,1}})e^{C\int_0^t||U^n||_{B^{\frac{3}{2}}_{2,1}}dt'}.
\end{aligned}\end{equation*}
Thanks to the induction hypothesis, for all $t\in[0,T]$, we deduce that
\begin{equation*}\begin{aligned}
		||U^{n+1}||_{B^{\frac{3}{2}}_{2,1}}\le (||U_0||_{B^{\frac{3}{2}}_{2,1}}+CT||\theta||_{B^{\frac{3}{2}}_{2,1}})\exp\bigg(\int_0^t{2C(||U_0||_{B^{\frac{3}{2}}_{2,1}}+C_0||\theta||_{B^{\frac{3}{2}}_{2,1}})}dt'\bigg),
\end{aligned}\end{equation*}
so that choosing $C_0\ge C$, where $C$ is the constant that appears in the above inequality, we deduce that
\begin{equation*}
	||U^{n+1}||_{L^{\infty}([0,T];B^{\frac{3}{2}}_{2,1})}\le 2\bigg(||U_0||_{B^{\frac{3}{2}}_{2,1}}+C_0||\theta||_{B^{\frac{3}{2}}_{2,1}}\bigg).
\end{equation*}
Therefore, $(U^n)_{n\in\mathbb N}$ is bounded in $L^{\infty}([0,T],B^{\frac{3}{2}}_{2,1})$.\\
$\bf{Step~3:~Convergence}$

 From now on, we are going to show that $(U^n)_{n\in\mathbb N}$ is a Cauchy sequence in $L^{\infty}([0,T];L^2)$. By \eqref{2.1}, we have
\begin{equation*}
	\partial_t(U^{n+1}-U^n)+A(U^n)\partial_x(U^{n+1}-U^n)=-\big(A(U^{n+1})-A(U^n)\big)\partial_x(U^n).
\end{equation*}
By the standard energy estimate, we obtain
\begin{equation*}\begin{aligned}
		||U^{n+1}-U^n||_{L^{\infty}_T(L^2)}\le&||\Delta_n U_0||_{L^2}\\
		 &+CT\bigg(||U_0||_{B^{\frac{3}{2}}_{2,1}}+||\theta||_{B^{\frac{3}{2}}_{2,1}}\bigg)\bigg(||U^{n+1}-U^n||_{L^{\infty}_T(L^2)}+||U^n-U^{n-1}||_{L^{\infty}_T(L^2)}\bigg).
\end{aligned}\end{equation*}
Define $V_{n+1}=||U^{n+1}-U^n||_{L^{\infty}_T(L^2)}$. Assuming that $4CT(||U_0||_{B^{\frac{3}{2}}_{2,1}}+||\theta||_{B^{\frac{3}{2}}_{2,1}})\le1$, we can get
\begin{equation*}
	V_{n+1}\le\frac{4}{3}||\Delta_n U_0||_{L^2}+\frac{1}{3}V_n.
\end{equation*}
As $||\Delta_n U_0||_{L^2}\le C2^{-ns}$, this implies that $\sum V_n$ is convergent. Then we can deduce that $(U^n)_{n\in\mathbb N}$ is a Cauchy sequence in $L^{\infty}([0,T];L^2)$, thus an interpolation argument ensures that $(U^n)_{n\in\mathbb N}$ is a Cauchy sequence in $L^{\infty}([0,T];B^{s}_{2,1})$ for any $s<\frac{3}{2}$. The limit $U$ of $(U^n)_{n\in\mathbb N}$ is obviously a solution of \eqref{1.4}. By the Fatou property for the Besov space $B^{\frac{3}{2}}_{2,1}$, we obtain that $U$ belongs to $L^{\infty}([0,T];B^{\frac{3}{2}}_{2,1})$. By virtue of \eqref{1.4}, we get that $\partial_tU\in L^{\infty}([0,T];B^{s}_{2,1})$, then $U\in \mathcal{C}([0,T];B^{s}_{2,1})$. In summary, we see that $U$ belongs to $L^{\infty}([0,T];B^{\frac{3}{2}}_{2,1})\cap\mathcal{C}([0,T];B^s_{2,1})\cap \mathcal{C}^1([0,T];B^{s-1}_{2,1})$ for any $s<\frac{3}{2}$.\\
$\bf{Step~4:~Existence}$

Now we check that $U\in\mathcal{C}([0,T];B^{\frac{3}{2}}_{2,1})$. By lemma \ref{existence}, $\Delta_j U$ satisfies
\begin{equation*}
	\left\{
	\begin{aligned}
		&\partial_t\Delta_j U+(\overline{S}_{j-1}A(U))\partial_x\Delta_j U=R_j+\Delta_j M,\\
		&\Delta_j U|_{t=0}=\Delta_j U_0,
	\end{aligned}
	\right.
\end{equation*}
with
\begin{equation*}
	||R_j^n||_{L^2}\le Cc_j(t)2^{-\frac{3}{2}j}||\nabla U(t)||_{L^{\infty}}||\nabla U(t)||_{B^{\frac{1}{2}}_{2,1}},
\end{equation*}
\begin{equation*}
	||\Delta_j M||_{L^2}\le Cc_j(t)2^{-\frac{3}{2}j}||\theta||_{B^{\frac{3}{2}}_{2,1}}.
\end{equation*}
\iffalse
By an $L^2$ estimate and time integration, we have
\begin{equation*}
	||\Delta_j U(t)||^2_{L^{2}}\le||\Delta_j U_0||_{L^{2}}+C 2^{-3j}\int_0^t c^2_j(t')\big(||\nabla U(t')||_{L^{\infty}}||\nabla U(t')||_{B^{\frac{1}{2}}_{2,1}}+||\theta||_{B^{\frac{3}{2}}_{2,1}}\big)dt'.
\end{equation*}
Multiplying both sides by $2^{3j}$ and then summing in $j$, we deduce that for all $t\in[0,T]$,
\fi
By a similar calculation as in $\bf{Step~2}$, we obtain
\begin{equation*}
	\sum_{j\ge-1}2^{\frac{3}{2}j}||\Delta_j U||_{L_t^{\infty}(L^{2})}\le 2\bigg(||U_0||_{B^{\frac{3}{2}}_{2,1}}+C_0||\theta||_{B^{\frac{3}{2}}_{2,1}}\bigg).
\end{equation*}
Then consider a positive $\varepsilon$, we can find some $j_0$ such that
\begin{equation*}
	\sum_{j\ge j_0}2^{\frac{3}{2}j}||\Delta_j U||_{L_T^{\infty}(L^{2})}\le \frac{\varepsilon}{4}.
\end{equation*}
Thus, we have
\begin{equation*}\begin{aligned}
		||U(t)-U(t')||_{B^{\frac{3}{2}}_{2,1}}\le&2\sum_{j\ge j_0}2^{\frac{3}{2}j}||\Delta_j U||_{L_T^{\infty}(L^{2})}+\sum_{j< j_0}2^{\frac{3}{2}j}||\Delta_j\big(U(t)-U(t')\big)||_{L_T^{\infty}(L^{2})}\\
		&\le C2^{\frac{3}{2}j_0}||U(t)-U(t')||_{L^{2}}+\frac{\varepsilon}{2}.
\end{aligned}\end{equation*}
From the above inequality and $\partial_tU\in \mathcal{C}^1([0,T];B^{\frac{1}{2}}_{2,1})$, we can deduce that $U$ belongs to $\mathcal{C}([0,T];B^{\frac{3}{2}}_{2,1})$. Thus, we get the existence of a solution for \eqref{1.4}.\\
$\bf{Step~5:~Uniqueness}$

We prove the uniqueness of the solution. Assume that $U$ and $V$ are two solutions of \eqref{1.4} with initial data $U_0$ and $V_0$, respectively. Then we have
\begin{equation*}
	\partial_t(U-V)+A(U)\partial_x(U-V)=-\big(A(U)-A(V)\big)\partial_x V.
\end{equation*}
The energy estimate and Gronwall's lemma ensure that
\begin{equation}\label{eq2}
	||U(t)-V(t)||_{L^2}\le||U_0-V_0||_{L^2}\exp\bigg(C\int_0^t{(||\nabla U(t')||_{L^{\infty}}+||\nabla V(t')||_{L^{\infty}})}dt'\bigg).
\end{equation}
From the above inequality, we prove the uniqueness.\\
$\bf{Step~6:~The~Continuous~Dependent}$

Assume that $U_0^n$ tends to $U_0^{\infty}$ in $B^{\frac{3}{2}}_{2,1}$, we find the solution $U^n$, $U^{\infty}$ of \eqref{2.1} with a common lifespan $T$. By $\bf{Step~1-Step~5}$, we have $U^n$, $U^{\infty}$ are uniformly bounded in $L^{\infty}([0,T],B^{\frac{3}{2}}_{2,1})$.

Taking advantage of the interpolation inequality, we see that $U^n\to U^{\infty}$ in $\mathcal{C}([0,T];B^{s}_{2,1})$ for any $s<\frac{3}{2}$. In order to prove $U^n\to U^{\infty}$ in $\mathcal{C}([0,T];B^{\frac{3}{2}}_{2,1})$, we only need to prove $\partial_x U^n\to\partial_x U^{\infty}$ in $\mathcal{C}([0,T];B^{\frac{1}{2}}_{2,1})$. Let $\overline{U}^n=\partial_x U^n$ and $\overline{U}^{\infty}=\partial_x U^{\infty}$. Note that $\overline{U}^n$ solves the following linear equation:
\begin{equation*}\label{4.1}
	\left\{
	\begin{aligned}
		&\partial_{t}\overline{U}^n+A(U^n)\partial_{x}\overline{U}^n=\partial_x(M(x))-\partial_x[A(U^n)]\partial_x\overline{U}^n,\\
		&\overline{U}^n|_{t=0}=\partial_{x}U^n_0.
	\end{aligned}
	\right.
\end{equation*}
Let $F^n=\partial_x(M(x))-\partial_x[A(U^n)]\partial_x\overline{U}^n$ and split $\overline{U}^n$ into $W^n+Z^n$ such that
\begin{equation*}
	\left\{
	\begin{aligned}
		&\partial_{t}Z^n+A(U^n)\partial_{x}Z^n=F^n-F^{\infty},\\
		&Z^n|_{t=0}=\partial_{x}U^n_0-\partial_{x}U^{\infty}_0
	\end{aligned}
	\right.
\end{equation*}
and
\begin{equation*}
	\left\{
	\begin{aligned}
		&\partial_{t}W^n+A(U^n)\partial_{x}W^n=F^{\infty},\\
		&W^n|_{t=0}=\partial_{x}U^{\infty}_0.
	\end{aligned}
	\right.
\end{equation*}
Because $(U^n)_{n\in\mathbb N}$ is bounded in $\mathcal{C}([0,T];B^{\frac{3}{2}}_{2,1})$, this implies that $(\partial_x A(U^n))_{n\in\mathbb N}$ is bounded in $\mathcal{C}([0,T];B^{\frac{1}{2}}_{2,1})$ and
\begin{equation*}
	F^n-F^{\infty}=-\partial_x\big[(A(U^n)-A(U^{\infty})\big]\partial_x U^n-\partial_xA(U^{\infty})\partial_x(U^n-U^{\infty}).
\end{equation*}
Then by the definition of $A$, Proposition \ref{embedding} and Lemma \ref{product}, we have
\begin{equation*}\begin{aligned}
		||F^n-F^{\infty}||_{B^{\frac{1}{2}}_{2,1}}&\le C\big(||\overline{U}^n||_{B^{\frac{1}{2}}_{2,1}}+||\overline{U}^{\infty}||_{B^{\frac{1}{2}}_{2,1}}\big)||\overline{U}^n-\overline{U}^{\infty}||_{B^{\frac{1}{2}}_{2,1}}\\
		&\le C\big(||\overline{U}^n||_{B^{\frac{1}{2}}_{2,1}}+||\overline{U}^{\infty}||_{B^{\frac{1}{2}}_{2,1}}\big)\big(||Z^n||_{B^{\frac{1}{2}}_{2,1}}+||W^n-W^{\infty}||_{B^{\frac{1}{2}}_{2,1}}\big)
\end{aligned}\end{equation*}
Next, according to Lemma \ref{priori estimate}, we obtain that for $t\in[0,T]$,
\begin{equation*}
	||Z^n(t)||_{B^{\frac{1}{2}}_{2,1}}\le \exp\big({C\int^t_0||\partial_x A(U^n)||_{B^{\frac{1}{2}}_{2,1}}d\tau}\big)\bigg(||V_0^n-V_0^{\infty}||_{B^{\frac{1}{2}}_{2,1}}+\int^t_0||F^n-F^{\infty}||_{B^{\frac{1}{2}}_{2,1}}d\tau\bigg).
\end{equation*}
It follows that
\begin{equation*}\begin{aligned}
		||Z^n(t)||_{B^{\frac{1}{2}}_{2,1}}\le& e^{CKt}\bigg(||V_0^n-V_0^{\infty}||_{B^{\frac{1}{2}}_{2,1}}\\
		 &+C\int^t_0\big(||V^n||_{B^{\frac{1}{2}}_{2,1}}+||V^{\infty}||_{B^{\frac{1}{2}}_{2,1}}\big)\times\big(||Z^n||_{B^{\frac{1}{2}}_{2,1}}+||W^n-W^{\infty}||_{B^{\frac{1}{2}}_{2,1}}\big)d\tau\bigg).
\end{aligned}\end{equation*}
Using the facts that

-- $(\overline{U}^n)_{n\in\mathbb N}$ is bounded in $\mathcal{C}([0,T];B^{\frac{1}{2}}_{2,1})$,

-- $\overline{U}_0^n$ tends to $\overline{U}^{\infty}_0$ in $B^{\frac{1}{2}}_{2,1}$,

-- $W^n$ tends to $W^{\infty}$ in $B^{\frac{1}{2}}_{2,1}$,\\
and then applying the Gronwall lemma, we conclude that $Z^n$ tends to 0 in $\mathcal{C}([0,T];B^{\frac{1}{2}}_{2,1})$.

Finally, we verify that
\begin{equation}\label{4.11}\begin{aligned}
		 ||\overline{U}_n-\overline{U}_{\infty}||_{L^{\infty}([0,T];B^{\frac{1}{2}}_{2,1})}&\le||Z^n-Z^{\infty}||_{B^{\frac{1}{2}}_{2,1}}+||W^n-W^{\infty}||_{B^{\frac{1}{2}}_{2,1}}\\
		&\le||Z^n||_{B^{\frac{1}{2}}_{2,1}}+||W^n-W^{\infty}||_{B^{\frac{1}{2}}_{2,1}},
\end{aligned}\end{equation}
which infers that
\begin{equation*}
	\partial_xU^n\to\partial_xU^{\infty}~~~\rm{in}~~~\mathcal{C}([0,T];B^{\frac{1}{2}}_{2,1}).
\end{equation*}
Hence, we prove the continuous dependence of \eqref{2.1} in $\mathcal{C}([0,T];B^{\frac{3}{2}}_{2,1})$.
Consequently, combining with $\bf{Step~1-Step~6}$, we finish the proof of Theorem \ref{them1}. $\hfill\Box$

Moreover, we prove a blow-up criterion.

\begin{theo}\label{thm3}
	\quad Let $U_0$ be in $B^{\frac{3}{2}}_{2,1}$, $\theta$ be in $B^{\frac{3}{2}}_{2,1}$ and let $T^*$ be the lifespan of the solution to \eqref{2.1}. If $T^*<\infty$, then we have
\begin{equation*}
	\int_0^{T^*}{||\partial_x U(t)||_{L^{\infty}}}dt=\infty.
\end{equation*}
\end{theo}

\begin{proof}\
\quad From \eqref{eq1}, after a few computations, we can obtain
\begin{equation}\begin{aligned}
	\frac{1}{2}\frac{d}{dt}||\Delta_j U||^2_{L^{2}}\le C||\partial_x U||_{L^{\infty}}||\Delta_j U||^2_{L^{2}}+C||R_j||_{L^2}||\Delta_j U||_{L^{2}}+||\partial_x\Delta_j\theta||_{L^{2}}||\Delta_j U||_{L^{2}},
\end{aligned}\end{equation}
which can be written as
\begin{equation}\begin{aligned}
\frac{d}{dt}||\Delta_j U||_{L^{2}}\le Cc_j(t)2^{-\frac{3}{2}j}||\partial_x U||_{L^{\infty}}||\partial_x U(t)||_{B^{\frac{1}{2}}_{2,1}}+Cc_j(t)2^{-\frac{3}{2}j}||\partial_x\theta||_{B^{\frac{1}{2}}_{2,1}}.
\end{aligned}\end{equation}
Multiplying both sides by $2^{\frac{3}{2}j}$, summing over $j$ and applying Gronwall inequality, we can get that
\begin{equation}\label{equu1}\begin{aligned}
   \|U\|_{B^{\frac{3}{2}}_{2,1}}\le\exp\big(\int_0^t\|\partial_xU\|_{L^{\infty}}dt'\big)\bigg(\|U_0\|_{B^{\frac{3}{2}}_{2,1}}+C\|\theta\|_{B^{\frac{3}{2}}_{2,1}}t\bigg).
\end{aligned}\end{equation}
If $T^*$ is finite and $\int_0^{T^*}||\partial_xU||_{L^{\infty}}dt<\infty$, then from \eqref{equu1}, we have $||U(t)||_{B^{\frac{3}{2}}_{2,1}}<\infty$ for all $t\in[0,T^*)$, which contradicts the assumption that $T^*$ is the maximal existence time.
%By \eqref{2.2}, we observe that the maximal time of existence $T^*$ satisfies
%\begin{equation}\label{3.1}
%	T^*>\frac{c}{||U_0||_{B^{\frac{3}{2}}_{2,1}}+C_0||\theta||_{B^{\frac{3}{2}}_{2,1}}}.
%\end{equation}
%Let $W$ be the solution of the Cauchy problem for \eqref{2.1} with data $U(t)$ at time $t$. From uniqueness, we know that $W(\tau)=U(t+\tau)$ for $0\le t+\tau<T^*$ %so that the maximal time of existence for $W$ is $T^*-t$. Hence, we get
%\begin{equation*}
%	T^*-t>\frac{c}{||U_0||_{B^{\frac{3}{2}}_{2,1}}+C_0||\theta||_{B^{\frac{3}{2}}_{2,1}}}.
%\end{equation*}
%which we can written as
%\begin{equation*}
%	||U_0||_{B^{\frac{3}{2}}_{2,1}}+C_0||\theta||_{B^{\frac{3}{2}}_{2,1}}>\frac{c}{T^*-t}.
%\end{equation*}
%This implies that $||U_0||_{B^{\frac{3}{2}}_{2,1}}+C_0||\theta||_{B^{\frac{3}{2}}_{2,1}}$ does not remain bounded when $t\to T^*$.
%By the uniqueness, the solution $W$ extends the solution $U$ beyond $T^*$. This contradicts the fact that $T^*$ is the lifespan.
\end{proof}

\section{Local well-posedness for $0<\beta<1$}
In this section, we give the proof of Theorem \ref{them2}.

Firstly, we will prove the existence. Consider the sequence $(U^n)_{n\in\mathbb N}$ defined by $U^0=0$ and
\begin{equation}\label{4.1}
	\left\{\begin{array}{l}
	D^{\beta}_{t}U^{n+1}+A(U^n)\partial_{x}U^{n+1}=M(x),\\
	U^{n+1}|_{t=0}=S_{n+1}U_0.
	\end{array}\right.
\end{equation}
We also claim that some constant $C_0$ can be found such that
\begin{equation*} T<\min\bigg\{\beta^{\frac{1}{\beta}},\bigg(\frac{\beta \ln2}{2C_0\big(||U_0||_{B^{\frac{3}{2}}_{2,1}}+C_0||\theta||_{B^{\frac{3}{2}}_{2,1}}\big)}\bigg)^{1/\beta}\bigg\}\Longrightarrow||U^n||_{L^{\infty}([0,T];B^{\frac{3}{2}}_{2,1})}\le 2\big(||U_0||_{B^{\frac{3}{2}}_{2,1}}+C_0||\theta||_{B^{\frac{3}{2}}_{2,1}}\big).
\end{equation*}
Similar to the calculation of $\beta=1$, we get
\begin{equation*}
	\frac{1}{2}D_t^{\beta}||\Delta_j U^{n+1}||^2_{L^{2}}\le C||\nabla U^n||_{L^{\infty}}||\Delta_j U^{n+1}||^2_{L^{2}}+C||R_j^n||_{L^2}||\Delta_j U^{n+1}||_{L^{2}}+||\partial_x\theta||_{L^{2}}||\Delta_j U^{n+1}||_{L^{2}}.
\end{equation*}
Then as $B^{\frac{1}{2}}_{2,1}\hookrightarrow L^{\infty}$, we have
\begin{equation*}\begin{aligned}
		D_t^{\beta}||\Delta_j U^{n+1}||_{L^{2}}\le &C||U^n||_{B^{\frac{3}{2}}_{2,1}}||\Delta_j U^{n+1}||_{L^{2}}\\
		&+C c_j(t)2^{-\frac{3}{2}j}||U^n||_{B^{\frac{3}{2}}_{2,1}}||U^{n+1}||_{B^{\frac{3}{2}}_{2,1}}+Cc_j(t)2^{-\frac{3}{2}j}||\theta||_{B^{\frac{3}{2}}_{2,1}}.
\end{aligned}\end{equation*}
As $||\Delta_j U^{n+1}||_{L^{2}}\le c_j(t)2^{-\frac{3}{2}j}||U^{n+1}||_{B^{\frac{3}{2}}_{2,1}}$, we thus get
\begin{equation*}
	D_t^{\beta}||\Delta_j U^{n+1}||_{L^{2}}\le C c_j(t)2^{-\frac{3}{2}j}||U^n||_{B^{\frac{3}{2}}_{2,1}}||U^{n+1}||_{B^{\frac{3}{2}}_{2,1}}+Cc_j(t)2^{-\frac{3}{2}j}||\theta||_{B^{\frac{3}{2}}_{2,1}}.
\end{equation*}
By Lemma \ref{fractional}, we obtain
\begin{equation*}
	||\Delta_j U^{n+1}||^2_{L^{2}}\le ||\Delta_j U_0||^2_{L^{2}}+C 2^{-\frac{3}{2}j}\int_0^{t}\tau^{\beta-1} c_j(\tau)\big(||U^n||_{B^{\frac{3}{2}}_{2,1}}||U^{n+1}||_{B^{\frac{3}{2}}_{2,1}}+||\theta||_{B^{\frac{3}{2}}_{2,1}}\big)d\tau.
\end{equation*}
Multiply both sides by $2^{\frac{3}{2}j}$ and then sum over $j$ to obtain
\begin{equation*}\begin{aligned}
		||U^{n+1}||_{L_T^{\infty}(B^{\frac{3}{2}}_{2,1})}&\le\sum_j2^{\frac{3}{2}j}||\Delta_j U^{n+1}||_{L_T^{\infty}(L^{2})}\\
		&\le ||U_0||_{B^{\frac{3}{2}}_{2,1}}+C\int_0^{t}\tau^{\beta-1}\big(||U^n||_{B^{\frac{3}{2}}_{2,1}}||U^{n+1}||_{B^{\frac{3}{2}}_{2,1}}+||\theta||_{B^{\frac{3}{2}}_{2,1}}\big)d\tau.
\end{aligned}\end{equation*}
The Gronwall lemma infers that
\begin{equation*}\begin{aligned}
		||U^{n+1}||_{B^{\frac{3}{2}}_{2,1}}\le (||U_0||_{B^{\frac{3}{2}}_{2,1}}+C\int_{0}^{t}\tau^{\beta-1}||\theta||_{B^{\frac{3}{2}}_{2,1}}d\tau)\exp\big(C{\int_0^t||U^n||_{B^{\frac{3}{2}}_{2,1}}\tau^{\beta-1}d\tau}\big),
\end{aligned}\end{equation*}
then thanks to the induction hypothesis, for all $t\in[0,T]$, we have
\begin{equation*}\begin{aligned}
		||U^{n+1}||_{B^{\frac{3}{2}}_{2,1}}\le \bigg(||U_0||_{B^{\frac{3}{2}}_{2,1}}+\frac{C}{\beta}||\theta||_{B^{\frac{3}{2}}_{2,1}}T^{\beta}\bigg)\exp\bigg(\int_0^t{2C(||U_0||_{B^{\frac{3}{2}}_{2,1}}+||\theta||_{B^{\frac{3}{2}}_{2,1}})\tau^{\beta-1}d\tau}\bigg).
\end{aligned}\end{equation*}
Choosing $C_0\le C$, we obtain that
\begin{equation*}
	||U^{n+1}||_{L^{\infty}([0,T];B^{\frac{3}{2}}_{2,1})}\le 2\bigg(||U_0||_{B^{\frac{3}{2}}_{2,1}}+C_0||\theta||_{B^{\frac{3}{2}}_{2,1}}\bigg).
\end{equation*}
By the same token, we can show that $(U_n)_{n\in\mathbb N}$ is a Cauchy sequence in $L^{\infty}([0,T];L^2)$. From the Fatou property, we get that $U$ belongs to $L^{\infty}([0,T];B^{\frac{3}{2}}_{2,1})$. By virtue of  the equation \eqref{1.4}, we have $D^\beta_t U$ belongs to $L^{\infty}([0,T];B^{\frac{1}{2}}_{2,1})$. The property of $D^\beta_t$ ensure that $U\in \mathcal{C}([0,T];B^\frac{3}{2}_{2,1})\cap \mathcal{C}^{\beta}([0,T];B^{\frac 1 2}_{2,1})$.
Consequently, we finish the proof of Theorem \ref{them2}.  $\hfill\Box$
\begin{rema}
From the above proof, we can see that the lifespan $$T\approx [\frac{\beta}{(\|U_0\|_{B^\frac{3}{2}_{2,1}}+\|\theta\|_{B^\frac{3}{2}_{2,1}})}]^{\frac{1}{\beta}}.$$
Consider the function $f(x)=(\frac{x}{C})^\frac{1}{x}$ with $x\in(0,1]$. Directly calculation yields
$$f'(x)=(\frac{x}{C})^\frac{1}{x}\bigg[\frac{1}{x}\ln(\frac{x}{C})\bigg]'=(\frac{x}{C})^\frac{1}{x}\frac{1-\ln x+\ln C}{x^2}.$$
If $C\geq e^{-1}$, then $f'(x)>0$. If $C< e^{-1}$, then there exists a critical point $x_0$ such that $f'(x_0)=0$. Moreover, $x_0$ is the maximal point of $f$. If the initial data is large enough, then the lifespan increases monotonically with increasing $\beta$. On the other hands, if the initial data is small, then there exists a $\beta$ such that the lifespan is largest. Therefore, one can find the best time-fractional model to study about Tsunamis when the initial data is small.
\end{rema}

\section{Blow up phenomena}
In this section, we study about the blow up phenomena. Our proof is based on the symmetric structure and the sign preserve property of the \eqref{1.1}.\\
{\bf{Proof of theorem \ref{them3} :}}\\
Let $[0,T)$ be the maximal interval of existence of the solution. Using the condition that $u_0$ is odd, $\psi_0, \theta$ are even, one can check that the functions
$$\widetilde{u}(x,t):=-u(-x,t), \quad \widetilde{\psi}(x,t):=\psi(-x,t), \quad t\in[0,T), x\in\mathbb{R},$$
are also the solutions of \eqref{1.1} with initial data $(u_0,\psi_0)$. By uniqueness we conclude that $\widetilde{u}=u$ and $\widetilde{\psi}=\phi$ which means that $u(\cdot,t)$ is odd and $\psi(\cdot,t)$ is even. In particular, we get
$$u(0,t)=u_{xx}(0,t)=0, \quad t\in[0,T).$$
Note that $\psi$ is even and therefore $\psi_x(0,t)=0$. Using the condition $\theta(0)=0$ and plugging $x=0$ into the second equation of \eqref{1.1}, we have
$$D^\beta_t\psi(0,t)+\psi(0,t)\phi_x(0,t)=0.$$
Since the initial data $\psi_0(0)=0$, it follows that $\psi(0,t)=0$ for any $t\in[0,T)$.
Applying $\partial^2_x$ to \eqref{1.1} and taking $x=0$, yields that
$$D^\beta_t\psi_{xx}(0,t)+\psi_{xx}(0,t)\phi_x(0,t)+\theta_{xx}(0)u_x(0,t)=0,$$
where we use the fact that $\psi(0,t)=\psi_x(0,t)=\psi_{xxx}(0,t)=\theta(0)=0$. Notice that $u'_{0}(0)<0$, by the continuity, there exists a $\ep>0$ such that
$u_x(0,t) \leq 0 $ for any $t\in[0, \ep)$. Combining with the condition $\theta_{xx}(0)\geq 0 $, we deduce that
$$D^\beta_t\psi_{xx}(0,t)+\psi_{xx}(0,t)\phi_x(0,t)\leq 0, \quad \forall t\in[0,\ep).$$
 The condition $\psi''_0(0)\leq 0$ ensures that $\psi_{xx}(0,t)\leq 0$ for $t\in[0,\ep)$. Applying $\partial_x$ to the first equation of \eqref{1.1}, we have
$$D^\beta_tu_x+ u^2_x+uu_{xx}+g\psi_{xx}=0.$$
Plugging $x=0$ into the above equation yields that
\begin{align}
D^\beta_tu_x(0,t)\leq -u^2_x(0,t) \quad \forall t\in[0,\ep).
\end{align}
By the definition of $D^\beta_t$, we have $u_x(0,\ep)\leq u_0(0)<0$. By virtue of the continuity argument, one can prove that
\begin{align}
D^\beta_tu_x(0,t)\leq -u^2_x(0,t) \quad \forall t\in[0,T).
\end{align}
Solving the above inequality, we deduce that
$$ \frac{1}{u_x(0,t)}-\frac{1}{u_{0x}(0)}\geq t^\beta,$$
which leads to
$$u_x(0,t)\leq \frac{u_{0x}(0)}{1+u_{0x}(0)t^\beta}.$$
Since $u_{0x}(0)<0$, it follows that $\lim_{t\rightarrow T^*}u_x(0,t)=-\infty$ where $T^*=(-\frac{1}{u_{0x}(0)})^{\frac{1}{\beta}}$.
$\hfill\Box$
\begin{rema}
Indeed, one can show that under some suitable symmetric condition on initial data at the point $x_0$, the blow up phenomena may occur at the point $x_0$.
\end{rema}
Now, we give two numerical examples to present the blow up phenomena. For simplicity, we take $\theta=0$ and $\beta=\frac{1}{2}$. Both of them satisfy the condition in theorem \ref{them3}.
\begin{exam}Set {~$u_0=-xe^{-x^2};\ \psi_0=0.02\times (cos(x/2+\pi)+1)$.} In this example, one can see that the solution will blow up at the point $x=0$ in finite time.\begin{figure}[!h]
       \begin{minipage}[t]{0.5\linewidth} % 图片占一行宽度0.5
               \centering
               \includegraphics[width=7.8cm,height=2.3cm]{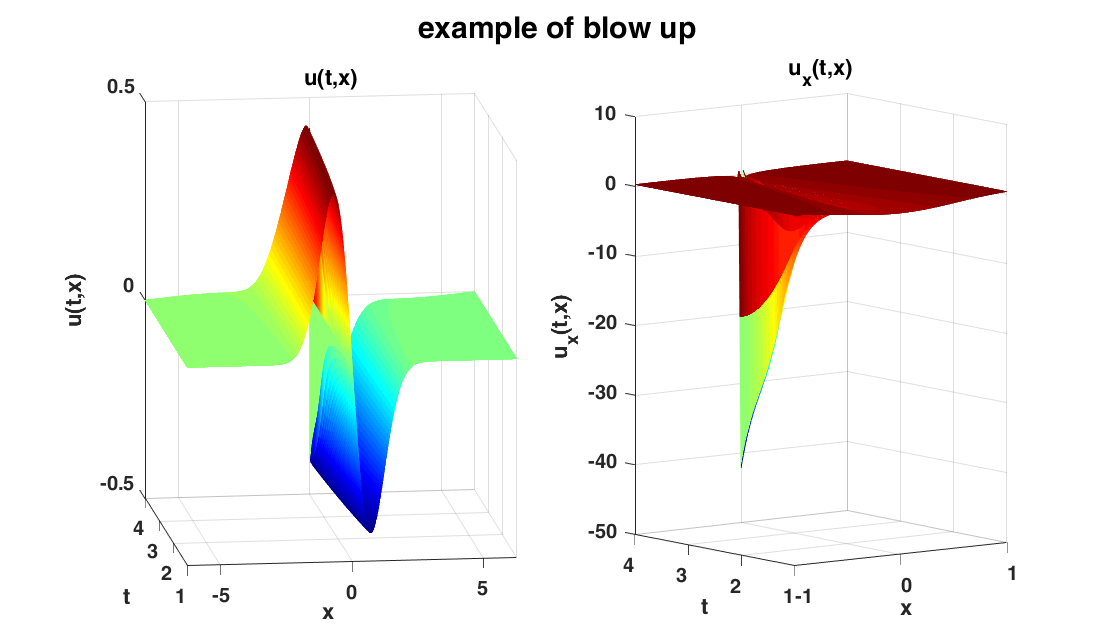}
               \caption{$u$ in example 1 of blow up}
        \end{minipage}
        \begin{minipage}[t]{0.5\linewidth} %图片占用一行宽度的50%
            \hspace{2mm}%微调2mm
            \includegraphics[width=7.8cm,height=2.3cm]{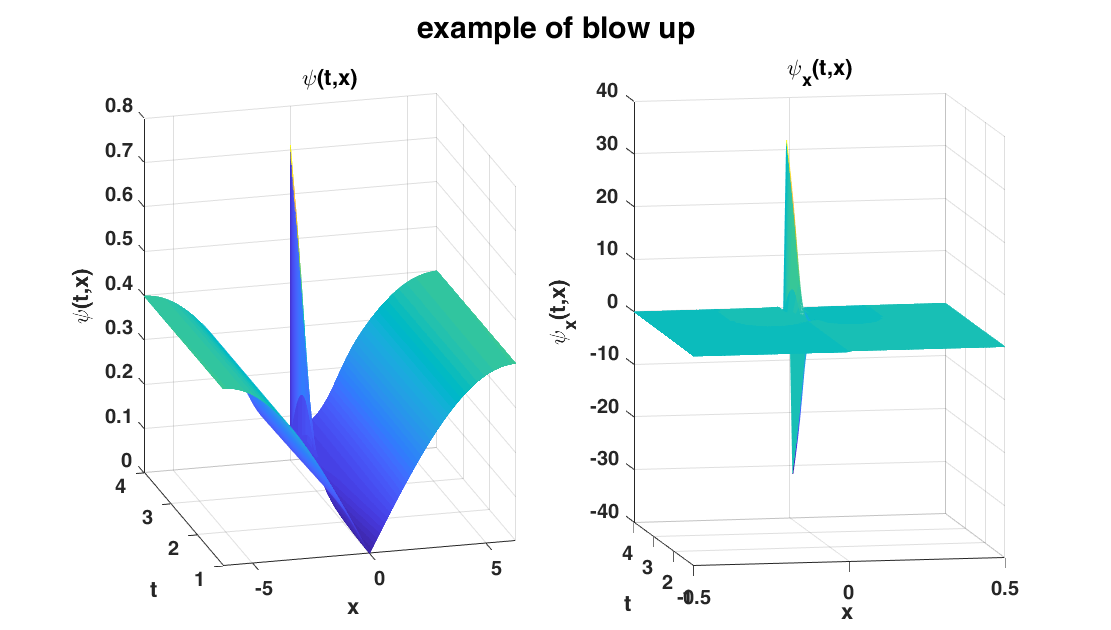}
            \caption{$\psi$ in example 1 of blow up}
         \end{minipage}
     \end{figure}\end{exam}
      \begin{exam}Set {~$u_0=\begin{cases}
      -\frac{2x}{(1-x^2)^2}e^{-\frac{1}{1-x^2}},~\text{for}~|x|\le1,\\
      0,~~~~~~~~~~~~~~~~~~~\text{for}~|x|>1, \
     \end{cases} \psi_0=0.02\times (cos(x/2+\pi)+1).
     $} In this example, we find that the solution will blow up at the point $x$ close to $0.6105$. This means that the blow up will occur before the time $T^*$ which we obtain in the proof of theorem \ref{them3}. Also, the example shows that the time fractional model \eqref{1.1} will lead to shock wave.
     \begin{figure}[h]
       \begin{minipage}[t]{0.5\linewidth} % 图片占一行宽度0.5
               \centering
               \includegraphics[width=7.8cm,height=2.3cm]{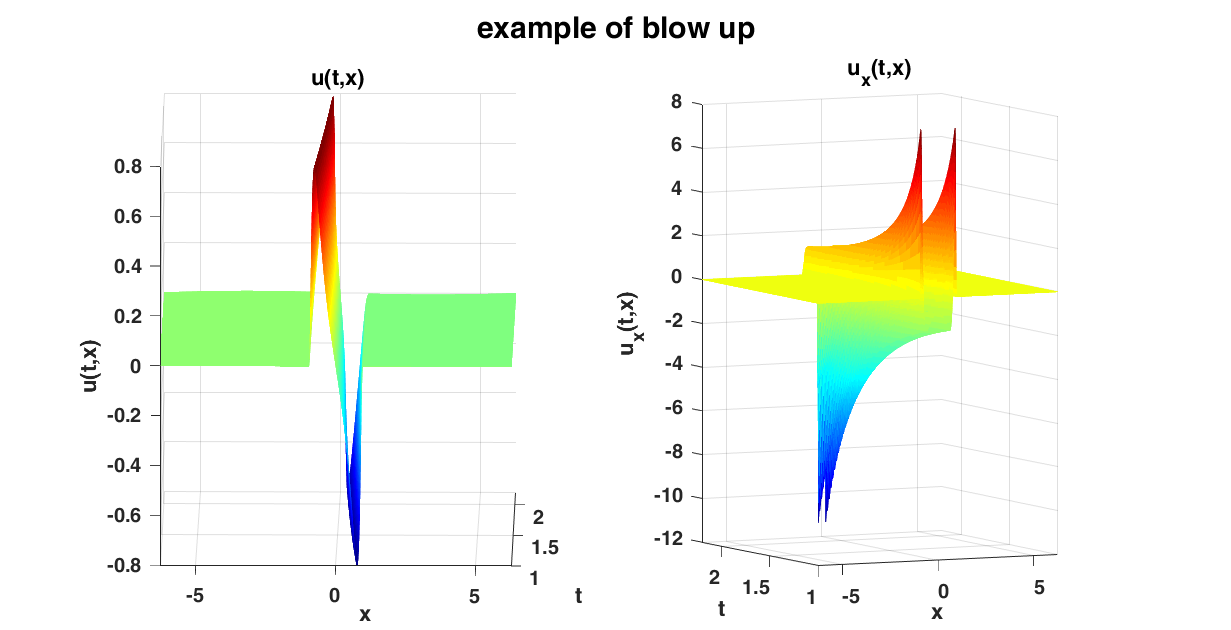}
               \caption{$u$ in example 2 of blow up}
        \end{minipage}
        \begin{minipage}[t]{0.5\linewidth} %图片占用一行宽度的50%
            \hspace{2mm}%微调2mm
            \includegraphics[width=7.8cm,height=2.3cm]{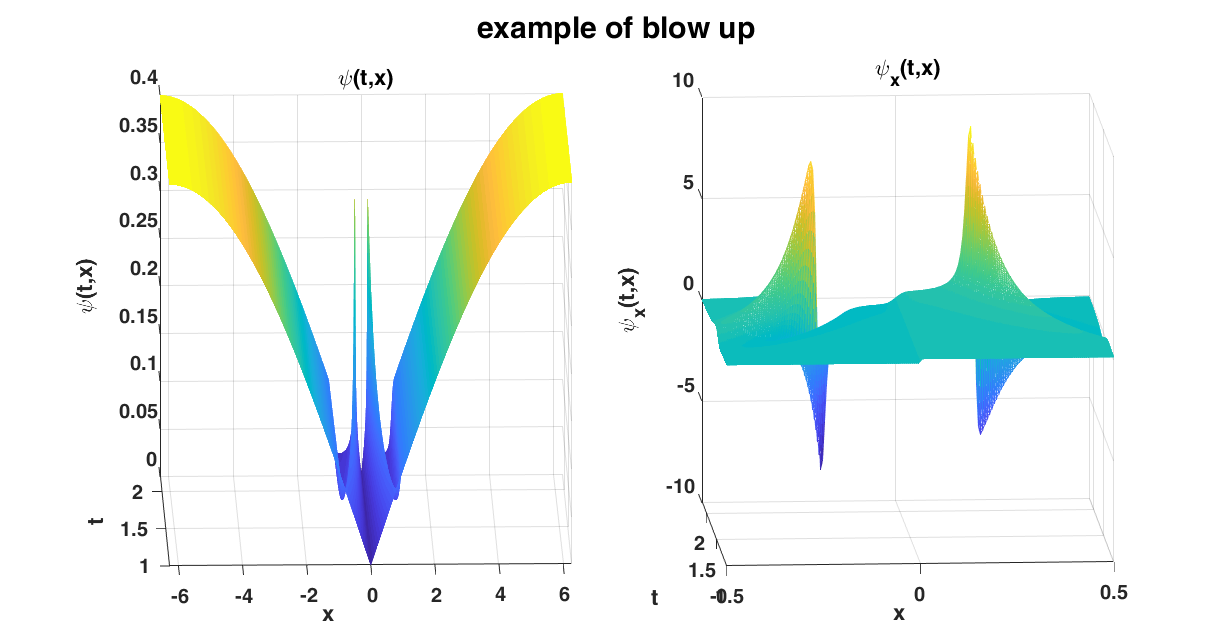}
            \caption{$\psi$ in example 2 of blow up}
         \end{minipage}
\end{figure}\end{exam}
\

\smallskip
\noindent\textbf{Acknowledgments} This work was
partially supported by the National key R\&D Program of China (2021YFA1002100), the National Natural Science Foundation of China (No.12171493), and the Natural Science Foundation of Guangdong province (No. 2021A1515010296 and 2022A1515011798 ).

\end{document}